\documentclass[11pt,reqno]{amsart}
\usepackage{amsmath, amssymb, amsthm}

\usepackage{epsfig}
\usepackage{graphicx}
\usepackage{color}
\usepackage[vcentermath]{youngtab}
\usepackage{young}
\usepackage{url}
\usepackage{nicefrac}
\usepackage{wrapfig}
\usepackage{amssymb}
\usepackage[breaklinks=true,pdfborder={0 0 0},pdfborderstyle={},colorlinks=true,linkcolor=black,citecolor=black,urlcolor=blue]{hyperref}

\allowdisplaybreaks

\DeclareMathOperator{\GL}{GL}

\DeclareMathOperator{\tr}{tr}

\DeclareMathOperator{\ch}{ch}

\newcommand{\PP}{{\mathbb P}}
\newcommand{\EE}{{\mathbb E}}

\newcommand{\Z}{\mathbb{Z}}

\newcommand{\R}{\mathbb{R}}
\newcommand{\C}{\mathbb{C}}

\usepackage{graphicx}
\usepackage{color}

\theoremstyle{plain}
\newtheorem{theorem}{Theorem}
\newtheorem*{theorem*}{Theorem}

\newtheorem{lemma}{Lemma}

\theoremstyle{definition}

\newtheorem{conjecture}{Conjecture}
\newtheorem*{conjecture*}{Conjecture}

\newtheorem{corollary}{Corollary}
\newtheorem*{rem*}{Remark}

\ifx\hyperlink\undefined

\else

\fi

\begin{document}

\title{The probability of long cycles in interchange processes}

\author{Gil Alon and Gady Kozma}

\begin{abstract}
We examine the number of cycles of length $k$ in a permutation, as a function on the symmetric group. We write it explicitly as a combination of characters of irreducible representations. This allows to study formation of long cycles in the interchange process, including a precise formula for the probability that the permutation is one long cycle at a given time $t$, and estimates for the cases of shorter cycles. 
\end{abstract}


\maketitle
\section{introduction}
A well known phenomenon in the theory of mixing times\footnote{We do
  not need the notion of mixing time in this paper, it is only used
  for comparison. The reader unfamiliar with it may peruse the
  survey \cite{MT06} or the book \cite{LPW09}.}  is that occasionally certain aspects of a system mix much
faster than the system as a whole. Pemantle \cite{P94} constructed an
example of a random walk on the symmetric group $S_n$ which mixes in
time $n^{1+o(1)}$ while every $k$ elements mix in $\le C(k)\sqrt{n}$
time. Schramm showed that for the interchange process on
the complete graph --- this is another random walk on $S_n$, see below
for details --- the structure of the large cycles mixes in time
$\approx n$, and it was known before \cite{DS81} that the mixing time
of this graph is $\approx n\log n$. See \cite{S05} and also \cite{B11}. Schramm's result is related to ---
in physics' parlance, it is the \emph{mean-field} case of --- a
conjecture of B\'alint T\'oth \cite{T93} that the cycle structure of
the interchange process on the graph $\Z^d$, $d\ge 3$, exhibits a
\emph{phase-transition}. In this paper we investigate the probability of long cycles, and obtain precise formulae for any graph, using the representation theory of $S_n$. As an application, we analyse certain
variations on T\'oth's conjecture.

Let us define the interchange process. Let $G$ be a finite graph with vertex set $\{1,\dotsc,n\}$, and equip each edge $\{i,j\}$ with an alarm clock
that rings with exponential rate $a_{i,j}$. Put a marble on every
vertex of $G$, all different, and whenever the clock of $\{i,j\}$ rings,
exchange the two marbles. Each marble therefore does a standard
continuous-time random walk on the graph but the different walks are
dependent. The positions of the marbles at time $t$ is a permutation
of their original positions, and viewed this way the process is a
random walk on the symmetric group. Note that we have changed the
timing from the previous paragraph. For example, if our graph is the
complete graph and $a_{i,j}=\nicefrac{1}{n}$ for all $i$ and $j$, then the process
mixes in time $\approx\log n$ and the large cycle structure
mixes in time $\approx 1$. However, the added convenience of
having each marble do the natural continuous time random walk
outweighs the difference in notations from some of the literature.

The stronger results of this paper require representation
theory to state, but let us start with two corollaries that can be
stated elementarily. Let $s_k(t)$ be the number of cycles of length
$k$ in our permutation at time $t$. Let
$0=\lambda_0\le\lambda_1\le\dotsb\le\lambda_{n-1}$ be the eigenvalues of
the continuous time Laplacian of the random walk on the graph $G$. Then
\begin{theorem}\label{thm:n}
We have $$ \PP(s_n(t)=1)=\frac1n\prod_{i=1}^{n-1}(1-e^{-\lambda_i t}) $$
\end{theorem}
Let us demonstrate the utility of this formula on the graph
$G=\{0,1\}^d$ with weights equal to 1. There is nothing particular about
this graph, but existing literature allows for easy comparison. For
example, Wilson \cite[\S 9]{W04} showed that the mixing time of the
interchange process on $G$ is $\ge cd$ (see also \cite{M06,O10}). The eigenvalues of $G$ may be
calculated explicitly: the eigenvectors are the Walsh functions, indexed by $y\in\{0,1\}^d$ and given by $f_y(x)=(-1)^{\sum_{i=1}^d x_i y_i}$.
We get that $2k$ is an eigenvalue with multiplicity ${d \choose k}$ for
$k=0,\dotsc,d$. Inserting into the formula at times
$\frac{1\pm\epsilon}{2}\log d$ gives
\begin{align*}
\PP\Big(s_n\Big(\frac{1-\epsilon}{2}\log d\Big)=1\Big)&=
2^{-d}\prod_{k=1}^d\left(1-e^{-(1-\epsilon)k\log d}\right)^{d \choose k}\le
\intertext{and looking only at $k=K:=\lfloor d^\epsilon/2\rfloor$,}
& \le \exp\left(-d^{(\epsilon-1)K}{d\choose K}\right)\stackrel{(*)}{\le} 
\exp\left(-\Big(\frac{d^{\epsilon}}{K}\Big)^K\right)\le 
\exp\left(-\exp\left(cd^\epsilon\right)\right)
\end{align*}
where $(*)$ comes from
\[
{d \choose
  K}=\Big(\frac{d}{K}\Big)^K\cdot\Big(\frac{1-\nicefrac{1}{d}}{1-\nicefrac{1}{K}}\cdot\frac{1-\nicefrac{2}{d}}{1-\nicefrac{2}{K}}\cdot\dotsb\Big)\ge\Big(\frac{d}{K}\Big)^K.
\]
On the other hand,
\begin{align*}
\PP\Big(s_n\Big(\frac{1+\epsilon}{2}\log d\Big)=1\Big) &= 
2^{-d}\exp\bigg(\sum_{k=1}^dO(d^{-(1+\epsilon)k}){d\choose k}\bigg) =
2^{-d}(1+O(d^{-\epsilon})).
\end{align*}
We see that the probability equilibrates at $\frac{1}{2}\log d$,
before the mixing time of the whole chain. Further, the equilibration
happens sharply --- this is reminiscent of the cutoff phenomenon for
mixing times. See \cite{DS81}, \cite{LS} or \cite[\S 18]{LPW09} for
the cutoff phenomenon. 

We remark that taking $t\to 0$ in Theorem \ref{thm:n} one can get a
new proof of Kirchoff's matrix-tree theorem. We fill the details in the appendix.

Another general, elementarily stated result is:
\begin{theorem}\label{thm:chuk} We have, for any graph $G$ and any $1\leq k \leq n$,
$$ \left|\EE(s_k(t))-\frac 1k\right| \leq \frac{3^n}{k}e^{-t\lambda_1} $$
\end{theorem}
The point about this result is its generality --- it holds for any
graph. In particular examples that we tried the estimate was worse
than the known or conjectured mixing time. But for general graphs it
seems to be the best known.

To proceed, let us recall a few basic facts about the representations
of $S_n$. For a full treatment see the books \cite{FH91,JK81,S00}. A representation of $S_n$ is a group homomorphism
$\tau:S_n\to \mathrm{GL}_k(\C)$ for some $k$, typically denoted by $\dim \tau$. Its character, denoted by $\chi_{\tau}$, is an
element of $L^2(S_n)$ defined by $\chi_{\tau}(g)=\tr(\tau(g))$. 
Now, the irreducible representations of $S_n$ are indexed
by partitions of $n$, namely, by sequences
$\lambda=[\lambda_1,\lambda_2,\dotsc ,\lambda_k]$ with $\lambda_1\ge
\lambda_2\ge \dotsb\ge \lambda_k>0$ and $\sum_{i=1}^k \lambda_i=n$ (we
denote this by $n \vdash \lambda$). A nice graphical representation of
partitions is using \emph{Young diagrams}, i.e.\ drawing each $\lambda_i$ as a line of boxes from top
to bottom, e.g.
\[
[5,1]={\tiny\yng(5,1)}\qquad [3,2,1]={\tiny\yng(3,2,1)} \qquad [2,1^3]={\tiny\yng(2,1,1,1)}.
\]
To each partition $n \vdash \lambda$ (and hence, for each young diagram with $n$ boxes) corresponds an irreducible representation, which we shall denote by
$U_\lambda$. For brevity, we denote the character of $U_\lambda$ by $\chi_\lambda$. 
Fix now some $1 \leq k \leq n$ and define 
\begin{equation}\label{eq:defalpha}
\alpha_k(g)=\#\{\textrm{cycles of length $k$ in $g$}\}.
\end{equation}
Now, $\alpha_k(g)$ depends only on the cycle structure of $g$, i.e.\ is
a class function, and hence it is a linear combinations of characters
of irreducible representations. Our main result is the precise
decomposition.

\begin{theorem}\label{thm:k} For any $n$ and $k$, 
\[
\alpha_k=\frac{1}{k}\sum_{n\vdash \rho}a_\rho \chi_\rho,
\]
where
\begin{equation}\label{eq:defak}
a_\rho=\begin{cases}
1&\rho=[n]\\
(-1)^{i+1} & \rho=[k-i-1,n-k+1,1^i]\mbox{ for some
}i\in\{0,\dotsc,2k-n-2\}\\
(-1)^i & \rho=[n-k,k-i,1^i]\mbox{ for some
}i\in\{\max\{2k-n,0\},\dotsc,k-1\}\\
0&\mbox{otherwise}
\end{cases}
\end{equation}
\end{theorem}
Let us describe this verbally (ignoring the diagram $[n]$ which has a somewhat special role). If $k>(n+1)/2$, start with $[k-1, n-k+1]$, with a minus sign. Now
drop boxes from the first row into the leftmost column until the first and second row are equal. Then drop in a single step two boxes, one from each of the first two rows to the leftmost column. Then start dropping
boxes from the second row until you reached a hook-shaped diagram. The sign keeps changing in each step.
If $k\le n/2$ start with the diagram $[n-k,k]$ with a plus sign, and drop boxes from the second row to the leftmost column until reaching a hook-shaped diagram, again switching sign at
each step. The case $k=(n+1)/2$ is similar except you start from
$[n-k,k-1,1]$ with a minus sign.

It is now clear what is special in the case $k=n$. In this case only
hook-shaped diagrams appear in the sum. For the hook-shaped diagrams
there is an explicit formula for the relevant eigenvalues discovered
by Bacher \cite{B94} (see also the appendix of \cite{AK09}). Let us remark that for $k<n$ the probability $\PP(s_k(t)=1)$ is not a function of the eigenvalues of the graph. In other words, one may find two \emph{isospectral} graphs for which these probabilities differ. We will
explain both facts (i.e.\ the conclusion of Theorem \ref{thm:n} from
Theorem \ref{thm:k} and the isospectral examples) in section
\ref{sec:cor} below. T\'oth's conjecture will be stated and discussed
in section \ref{sec:Toth}. 

We remark that Theorem \ref{thm:k} strengthens results by Eriksen and
Hultman \cite[\S 5]{EH04} who found the decomposition of $\sum\alpha_k$,
i.e.\ of the number of cycles of a permutations. The formulas of
\cite{EH04} are quite short and reveal some patterns in the numbers
$a_\rho$. For example, for every $\rho$ of the form $[a,b,1^c]$, $a_\rho\ne 0$ for exactly two
values of $k$, with opposite signs.

\section{Notations and preliminaries}\label{sec:notations}
Let $A=\{a_{i,j}\}_{1\le
  i<j\le n}$ be a collection of non-negative numbers which we consider
as a weighted graph. The random walk on $S_n$ associated with the weighted
graph $A$ is a process in continuous time starting from the identity
permutation $\mathbf{1}$ on $S_n$ and going from $g$ to $(ij)g$ with
rate $a_{i,j}$. Formally, consider $L^2(S_n)$, both as a Hilbert space with the standard inner product, and as an $\mathbb R$-algebra, via the \emph{group ring} structure.
Define the Laplacian as the element of
$L^2(S_n)$ given by
\[
\Delta=\Delta_A=\sum_{i<j}a_{i,j}(\mathbf{1}-(ij))
\]
where $\mathbf{1}$ is the element of $L^2(S_n)$ equal to 1 in the
identity permutation, and 0 everywhere else; and $(ij)$ is similarly a
singleton at the transposition $(ij)$. The distribution of the
location of our process at time $t$ is 
\[
e^{-t\Delta}=\sum_{k=0}^\infty\frac{(-t\Delta)^k}{k!}
\]
In particular for $\alpha_k$
defined by (\ref{eq:defalpha}),
\[
\EE(s_k(t))=
\sum_{g \in S_n}\left(e^{-t\Delta}\right)(g)\alpha_k(g)=n!\langle e^{-t\Delta},\alpha_k\rangle
\]
where here and below $\langle\cdot,\cdot\rangle$ stands for the
standard inner product in $L^2(S_n)$, i.e.\ $\langle a,b\rangle =
\nicefrac{1}{n!}\sum_{g\in S_n} \linebreak[4]a(g)\overline{b(g)}$.

For the proof of Theorem \ref{thm:k} we will need a second set of representations of $S_n$, this time 
\emph{reducible} representations. For $n \vdash \rho$, let $T_\rho<S_n$ be the
subgroup of all permutations fixing the sets $\{1,\dotsc,\rho_1\}$,
$\{\rho_1+1,\dotsc,\rho_1+\rho_2\}$, etc\@. As a group
$T_\rho\cong S_{\rho_1}\times \dotsb\times S_{\rho_r}$. Now,
$S_n$ acts on the left cosets of $T_\rho$, i.e.\ $\{hT_\rho\}_{h\in S_n}$, and using these cosets as a basis 
we obtain a representation of $S_n$, which we will denote by $V_\rho$. 
Readers familiar with exclusion
processes might find it convenient to think about $V_\rho$ as $\mathbb R^X$ where $X$ is the space of configurations of
the exclusion process with $\rho_1$ particles of colour 1, $\rho_2$
particles of colour 2 etc.\ --- considering $\Delta$ as an operator on
$V_\rho$ it is easy to verify that one gets an identical
process. We denote
\begin{equation}\label{eq:defVk}
\psi_\rho=\chi_{V_\rho}.
\end{equation}
It is well known that the representations $V_\rho$ are generally
reducible and their irreducible components, consist of all $U_\sigma$ for
$\sigma\trianglerighteq\rho$, where $\trianglerighteq$ is the
\emph{domination} order \cite[Corollary 4.39]{FH91} --- we say that $\sigma\trianglerighteq\rho$
when you can reach $\rho$ from $\sigma$ by a series of ``toppling'' of
a box of the Young diagram to a lower row which keep the structure
of a Young diagram. Alternatively, $\sigma\trianglerighteq\rho$ is equivalent to
\[
\sum_{i=1}^j\sigma_i \ge \sum_{i=1}^j\rho_i\qquad\forall j.
\]


\section{Character decomposition}

In this section we prove Theorem \ref{thm:k}. We go about it by describing a more general method for expressing a class function on $S_n$ as a linear combination of characters, and then applying it to our case.

Given a function $f:S_n\rightarrow \R$ that is a class function (i.e.\ satisfies $f(hgh^{-1})=f(g)$ for all $g,h\in S_n$), it can be expressed as a linear combination of the characters of $S_n$ (see, e.g., \cite[Proposition 2.30]{FH91}).
By the character orthogonality relations (ibid.), we have
\begin{equation}
f = \sum_{n \vdash \rho} \langle f, \chi_\rho \rangle \chi_\rho
\label{eq:orthog}
\end{equation}

As it is often hard to calculate the inner products $\langle f, \chi_\rho \rangle$ directly, we start by calculating $\langle f, \psi_\lambda \rangle$, where $\psi_\lambda=\chi_{V_\lambda}$ and $V_\lambda$ are the
``exclusion-like'' reducible representations defined just before
(\ref{eq:defVk}).
\begin{lemma} \label{lemma:average}
We have $\langle f,  \psi_\lambda \rangle=\frac1{\#(T_\lambda)} \sum_{q\in T_\lambda} f(q) $.
\end{lemma}

\begin{proof} We have $$ \langle f, \psi_\lambda \rangle = \frac1{n!}\sum_{g \in S_n} \psi_\lambda(g) f(g) $$
Recall from \S \ref{sec:notations} that $V_\lambda$ is obtained from the action of
$S_n$ on the cosets of a $T_\lambda<S_n$. By the definition of trace, $\psi_{\lambda}(g)$ equals the number of cosets of $T_{\lambda}$ fixed by $g$. A coset $h T_\lambda$ is fixed by $g$ iff $h^{-1} g h \in T_\lambda$. Hence,

$$ \langle f, \psi_\lambda \rangle = \frac1{n!} \sum_{hT_\lambda \in S_n/T_\lambda}\; \sum_{g: h^-1gh \in T_\lambda} f(g). \label{eq:youngrule} $$ 
Let us make a change of variables, $q = h^{-1} g h $. Since $f$ is a class function, we have

\[ \langle f, \psi_\lambda \rangle = \frac1{n!\#(T_\lambda)}
\sum_{h\in G} \sum_{q\in T_\lambda} f(hqh^{-1}) =
\frac1{n!\#(T_\lambda)} \sum_{h\in G} \sum_{q\in T_\lambda} f(q) =
\frac1{\#(T_\lambda)} \sum_{q\in T_\lambda} f(q). 
\qedhere
\]
\end{proof}

Now, by Young's rule \cite[Corollary 4.39]{FH91}, the characters $\psi_\lambda$ and the characters $\chi_\lambda$ are related by the linear equations 
$$ \psi_\lambda = \sum_{n \vdash \mu} K_{\mu \lambda}\chi_\mu $$
Where the numbers $K_{\mu \lambda}$, called the Kostka numbers, are defined as follows:
Let $\lambda=[\lambda_1,\dotsc,\lambda_r]$, then $K_{\mu \lambda}$ is the number of ways the Young diagram $\mu$ can be filled with $\lambda_1$ $1$'s, $\lambda_2$ $2$'s, etc., such that each row is nondecreasing, and each column is strictly increasing. The numbers $K_{\mu \lambda}$
 satisfy $K_{\mu \lambda}=0$ whenever $\mu < \lambda$ (with respect to the lexicographic order), and $K_{\mu \mu}=1$. (See \cite{FH91}, appendix A). Hence,
$$ \langle f,\psi_\lambda \rangle = \sum_{n \vdash \mu} K_{\mu \lambda}\langle f,\chi_\mu \rangle $$
In other words the numbers $\langle f,\chi_\mu\rangle$ satisfy
a system of linear equations, whose coefficient matrix $(K_{\mu \lambda})$
is triangular with $1$'s on the diagonal, hence invertible. 

The resulting system of equations has a more elegant form when expressed in terms of \emph{symmetric polynomials}.

Fix an integer $m \geq n$ (whose value is not important), and consider
the ring of symmetric polynomials in $m$ variables $x_1,\dotsc ,x_m$
over $\C$. Consider the following homogeneous symmetric polynomials of degree $n$ (see \cite{FH91}, ibid. for more details):

\begin{itemize}
\item For $n \vdash \lambda=[\lambda_1,\dotsc,\lambda_r]$, $M_\lambda = \sum_\alpha x^\alpha$, where $\alpha=(\alpha_1,\dotsc,\alpha_n)$ goes over all the possible permutations of $(\lambda_1,\dotsc \lambda_r,0,\dotsc,0)$.
\item The Schur polynomials $S_\mu= \sum_{\lambda}K_{\mu \lambda}M_{\lambda}$
\item The full homogeneous polynomial $H_n$, defined as the sum of all monomials of degree $n$. It is easy to see that for all $n \vdash \lambda$, $K_{[n]\lambda}=1$. Hence, $H_n = \sum_{n \vdash \lambda} M_\lambda = \sum_{n \vdash \lambda} K_{[n]\lambda} M_\lambda = S_{[n]}$.
\end{itemize}

Recall also the Frobenius characteristic map $\ch$, defined on the class functions of $S_n$, which sends an irreducible character $\chi_{\mu}$ to its corresponding Schur polynomial $S_{\mu}$, and is extended by linearity. Clearly, decomposing a class function into irreducible characters, $f=\sum_{\mu}a_\mu \chi_{\mu}$ is equivalent to decomposing its image $\ch(f)$ into Schur polynomials, $\ch(f)=\sum_{\mu} a_{\mu} S_{\mu}$. By (\ref{eq:youngrule}), we have

$$ \ch(f)=\sum_{\mu} \langle f, \chi_{\mu} \rangle S_{\mu} =
   \sum_{\mu} \langle f, \chi_{\mu} \rangle \sum_{\lambda} K_{\mu\lambda} M_{\lambda} =
   \sum_{\lambda} \langle f, \psi_{\lambda} \rangle M_\lambda $$
   
We conclude:

\begin{lemma} \label{lemma:ch} Let $f$ be a class function on $S_n$. Then $$\ch(f) = \sum_{\lambda} \left ( \frac{1}{\#(T_{\lambda})} \sum_{g\in T_{\lambda}} f(g) \right ) M_{\lambda}. $$

\end{lemma}

We now apply this to the class functions $\alpha_k$ (Recall the definition of $\alpha_k$, (\ref{eq:defalpha})).  

\begin{lemma} \label{lemma:ch_formula}
We have for all $1\leq k \leq n$, $\ch(\alpha_k)=\frac1k(\sum_{i=1}^m x_i^k)H_{n-k}(x_1,\dotsc,x_m)$.
\end{lemma}

\begin{proof} Let us define a function $\beta_k$ on the set of partitions of $n$ by 
$$\beta_k([\lambda_1,\dotsc,\lambda_r])= \#\{i: \lambda_i \geq k \}. $$
By lemma \ref{lemma:average}, 
$$ \langle \alpha_k, \psi_\lambda \rangle =  \frac1{\#(T_\lambda)} \sum_{q\in T_\lambda} \alpha_k(q). $$
The sum $\sum_{q\in T_\lambda} \alpha_k(q) $ can be evaluated by summing over all possible $k$-cycles $c\in T_\lambda$, the number of elements of $T_\lambda$ such that $c$ is one of their cycles. For any $i$ such that $\lambda_i\geq k$, there are $\binom{\lambda_i}{k} \cdot (k-1)!$ choices for a cycle $c$ in the $S_{\lambda_i}$-factor of $T_\lambda$, and $\lambda_1! \lambda_2! \dotsb (\lambda_i - k)! \dotsb \lambda_r!$ choices for an element $g\in T_\lambda$ with $c$ as a cycle. Hence each such $i$ contributes to the sum
$$ \binom{\lambda_i}{k} \cdot (k-1)! \cdot \lambda_1! \lambda_2! \dotsb (\lambda_i - k)! \dotsb \lambda_r! = \frac {\#(T_\lambda)} k$$ Obviously, if $\lambda_i < k$ then there are no $k$-cycles in the $S_{\lambda_i}$-factor, and the contribution is $0$. Hence,
\[
\langle \alpha_k, \psi_\lambda \rangle= \frac1{\#(T_\lambda)} \sum_{i:\lambda_i
  \geq k} \frac {\#(T_\lambda)} k = \frac 1 k \beta_k(\lambda).
\]
By lemma \ref{lemma:ch},
$$ \ch(\alpha_k)=\frac1k\sum_{\lambda} \beta_k(\lambda)M_{\lambda}. $$
A moment's reflection shows that  $ \sum_{\lambda} \beta_k(\lambda)M_{\lambda} = 
(\sum_{i=1}^m x_i^k)H_{n-k}$. Indeed, each monomial
$x_1^{\alpha_1}\dotsb \discretionary{}{\mbox{$\cdot\,$}}{}
 x_m^{\alpha_m}$ of degree $n$ appears on the left-hand side
with coefficient $\#\{i:\alpha_i\ge k\}$ (by the definition of
$\beta$), and the same is on the right-hand side. This finishes the lemma.
\end{proof}

Our goal is to express $\ch(\alpha_k)$ as a linear combination of Schur
polynomials. Let us start with the case of $k=n$. 

\begin{lemma} \label{lemma:alpha_n} $\ch(\alpha_n) = \frac1n\sum_{i=0}^{n-1} (-1)^i S_{[n-i,1^i]} $.
\end{lemma}

\begin{proof} By lemma \ref{lemma:ch_formula}, $\ch(\alpha_n)=\frac 1n \sum_i x_i^n = \frac1n M_{[n]}$. On the other hand, for all $0\leq i \leq n-1$ we have
$$ S_{[n-i,1^i]} = \sum_{\lambda} K_{[n-i,1^i]\lambda} M_{\lambda}. $$
Let $\lambda$ have $r$ rows. By definition of the Kostka numbers, we have $K_{[n-i,1^i]\lambda}=\binom{r-1}{i}$, and $K_{[n-i,1^i]\lambda}=0$ for $i\geq r$, since the top left box of $[n-i,1^i]$ has to be numbered $1$, and the whole configuration is determined by the choice of distinct $i$ numbers out of $2,\dotsc,r$ to be placed in the leftmost column in ascending order. 
Denoting by $r(\lambda)$ the number of rows in $\lambda$, we get
$$ \sum_{i=0}^{n-1}(-1)^iS_{[n-i,1^i]} = \sum_{\lambda} M_{\lambda} \sum_{i=0}^{n-1} (-1)^i \binom{r(\lambda)-1}{i} $$
By the binomial identity, the inner sum is $0$ unless $r(\lambda)=1$,
i.e.\ $\lambda=[n]$, in which case the inner sum is $1$. We get $ \sum_{i=0}^{n-1}(-1)^iS_{[n-i,1^i]}= M_{[n]}$, as desired.
\end{proof}

\begin{rem*}Lemma \ref{lemma:alpha_n} can be proved more directly by
  using the Murnaghan-Nakayama rule \cite[Theorem 4.10.2]{S00} to
  express $\alpha_n$ as a linear combination of characters: for any
  $n\!\vdash\!\lambda$, the
  scalar product $\langle\chi_\lambda,\alpha_n\rangle$ is,
  up to a constant, the value of $\chi_\lambda$ at one specific permutation, namely a
  cycle of length $n$. The Murnaghan-Nakayama rule, when applied to
  such a cycle, takes a simple form. 
\end{rem*}  

We immediately conclude:

\begin{corollary} We have 
$$ \alpha_n = \frac1n \sum_{i=0}^{n-1} (-1)^i \chi_{[n-i, 1^i]} $$
which is Theorem \ref{thm:k} for $k=n$.
\end{corollary}

Let us now treat the general case, using the case we already
proved. By lemma \ref{lemma:alpha_n}, applied to $k$,
$$ \frac 1k \sum_{i=1}^m x_i^k = \frac 1k \sum_{i=0}^{k-1}(-1)^i S_{[k-i,1^i]} $$
Hence, by lemma \ref{lemma:ch_formula},
$$ \ch(\alpha_k) = \frac 1k \left( \sum_{i=0}^{k-1}(-1)^i S_{[k-i,1^i]} \right ) H_{n-k}. $$

We now apply Pieri's formula (see \cite{FH91}), according to which,
$S_{[k-i,1^i]} H_{n-k}$ is the sum of all polynomials of the form
$S_{\lambda'}$, where $\lambda'$ is obtained by adding $n-k$ boxes to
$[k-i,1^i]$, without adding two boxes in the same column. Since we
have a hook-shaped diagram, our possibilities are rather limited:
we may add a box at the leftmost column or not, and the rest of the
boxes go in the first two rows. Denote therefore
$$ S_{[k-i,1^i]}H_{n-k}=A_i+B_i$$
where $A_i$ is the sum when one does not add a square at the leftmost
column, and $B_i$ is when one does. Denote also
$x(i,j)=S_{[n-i-j,1+j,1^{i-1}]}$ (the contribution coming from adding $j$ boxes to the second row of $[k-i,1^i]$, and the remaining $n-k-j$ boxes to the first row). Then
\begin{gather*}
 A_0=S_{[n]}\qquad A_i = \sum_{j=0}^{\min(n-k,k-i-1)}x(i,j)\\
 B_i = \sum_{j=0}^{\min(n-k-1,k-i-1)}x(i+1,j).
\end{gather*}
We now sum over $i$ and get,
$$ \left(\sum_i x_i^k \right ) H_{n-k} = \sum_{i=0}^{k-1} (-1)^i (A_i+B_i)$$

Our next goal is to find the alternating sum $\sum_{i=0}^{k-1} (-1)^i
(A_i+B_i)$. There are further cancellations here because $B_i$ and
$A_{i+1}$ are quite similar --- $B_i$ corresponds to adding a box to
the first column of $[k-i,1^i]$ while $A_{i+1}$ corresponds to not
adding a box to the first column of $[k-i-1,1^{i+1}]$. Hence most of
the terms cancel out. We get
\begin{align*}
A_{i+1}&= \left \{ \begin{array}{ll}
\sum_{j=0}^{n-k}x(i+1,j) &   0 \leq i \leq 2k-n-2 \\
\sum_{j=0}^{k-i-2}x(i+1,j) &  2k-n-1 \leq i \leq k-2 \\ 
\end{array}
\right.\\
B_i&= \left \{ \begin{array}{ll}
\sum_{j=0}^{n-k-1}x(i+1,j) &   0 \leq i \leq 2k-n-1 \\
\sum_{j=0}^{k-i-1}x(i+1,j) & 2k-n \leq i \leq k-1 \\ 
\end{array}
\right.
\end{align*}
Hence (putting $A_k=0$),
$$ B_i-A_{i+1}= \left \{ \begin{array}{ll}
\sum_{j=0}^{n-k-1}x(i+1,j)-\sum_{j=0}^{n-k}x(i+1,j)=-x(i+1,n-k) & 0 \leq i \leq 2k-n-2 \\
\sum_{j=0}^{n-k-1}x(i+1,j)-\sum_{j=0}^{n-k-1}x(i+1,j)=0&  i=2k-n-1 \\  
\sum_{j=0}^{k-i-1}x(i+1,j)-\sum_{j=0}^{k-i-2}x(i+1,j)=x(i+1,k-i-1)& 2k-n \le i \leq k-1 \\ 
\end{array}
\right.
$$
and
\begin{align*}
 \sum_{i=0}^{k-1} (-1)^i (A_i+B_i)&=A_0+\sum_{i=0}^{k-1}(-1)^i(B_i-A_{i+1})=\\
&=S_{[n]}-\sum_{i=0}^{2k-n-2}(-1)^ix(i+1,n-k)+\sum_{i=2k-n}^{k-1}(-1)^ix(i+1,k-i-1)=\\
& = S_{[n]}-\sum_{i=0}^{2k-n-2}(-1)^iS_{[k-i-1,n-k+1,1^i]}+\sum_{i=2k-n}^{k-1}(-1)^i S_{[n-k,k-i,1^i]} = \sum_{\rho} a_{\rho} S_{\rho}
\end{align*}
where the numbers $a_{\rho}$ were defined in the statement of Theorem \ref{thm:k}.
Hence, by lemma \ref{lemma:ch_formula},
$$ k \cdot \ch(\alpha_k) = \left(\sum_i x_i^k \right ) H_{n-k} = \sum_{\rho} a_{\rho} S_{\rho}.$$
This ends the proof of Theorem \ref{thm:k}.\qed

\section{The probability of long cycles}\label{sec:cor}
Let $\rho$ be a partition of $n$, and
let $U_\rho:S_n\to\GL(\C^{\dim U_\rho})$ be the corresponding irreducible
representation. Let $D=\sum d_gg$ be any element of the group
ring. Then $U_\rho(D)$ is the element of $\GL(\C^{\dim U\rho})$ given by
\[
\sum_g d_gU_\rho(g).
\]
(it might be useful to think about $U_\rho(D)$ as a non-commutative Fourier transform
of $D$, with the fact that $U_\rho(D_1D_2)=U_\rho(D_1)U_\rho(D_2)$
being the non-commutative analog of
$\widehat{f*g}=\widehat{f}\widehat{\vphantom{f}g}$). 
In the case that $D=\Delta_A$ we will denote the eigenvalues of this matrix by
$0\leq\lambda_1(A,\rho)\leq \dotsc \leq \lambda_{\dim(\rho)}(A,\rho)$
(it is well-known that $U_\rho(\Delta_A)$ is positive semidefinite and
in particular diagonalizable, see e.g.\ \cite{AK09}).

\begin{lemma} \label{lemma:formula}
For any $n$ and $k$ we have 
\[
\EE(s_k(t))=\frac
1k\sum_{n\vdash\rho}a_\rho\sum_{j=1}^{\dim U_\rho}e^{-t\lambda_j(A,\rho)}
\]
where $a_\rho$ are as in Theorem \ref{thm:k}.
\end{lemma}

\begin{proof}
As discussed in \S\ref{sec:notations}, 
$$\EE(s_k(t))=n!\langle \alpha_k, e^{-t\Delta_A} \rangle = \frac1k \sum_\rho a_\rho n!\langle e^{-t\Delta_A}, \chi_\rho\rangle $$
By definition, $\chi_\rho$ attaches to each $g\in S_n$ the trace of
$g$ acting on the representation $U_\rho$. By the linearity of the trace,
\[
n!\langle e^{-t\Delta_A},\chi_\rho\rangle =
\tr\left(U_\rho\left(e^{-t\Delta_A}\right)\right)
\]
where $U_\rho(\cdot)$ is the action of a representation on an
element of the group ring as above. Further, for every representation
$U$ and any element $D$ of the group ring,
\[
U(e^D)=e^{U(D)}
\]
where the exponentiation on the left-hand side is in the group ring
while on the right-hand side we have exponentiation of matrices. Since $U_\rho(-t \Delta)$ is diagonalizable,

\[
\tr\left(U_\rho\left(e^{-t\Delta}\right)\right)=\sum_j
e^{-t\lambda_j(A,\rho)}.
\]
The proof now follows from Theorem \ref{thm:k}.
\qedhere

\end{proof}
\begin{proof}[Proof of Theorem \ref{thm:n}]
$s_n(t)$ can take only the values 0 and 1. Hence, using lemma \ref{lemma:formula} for $k=n$, we get
$$ \PP(s_n(t)=1)=\EE(s_n(t))=\frac1n\sum_{i=0}^{n-1} (-1)^i \sum_j e^{-t\lambda_j(A, [n-i,1^i])} $$
Since $[n-i,1^i]$ is a hook-shaped diagram, the eigenvalues
$\lambda_j(A,[n-i,1^i])$ are simply all the sums of $i$-tuples of the
eigenvalues $\lambda_1(A),\dotsc,\lambda_{n-1}(A)$. (See \cite{B94}
and also the appendix of \cite{AK09}). Hence,
\[
\PP(s_n(t)=1)=\frac1n\left(1+\sum_{i=1}^{n-1}(-1)^i\sum_{1\leq j_1 <
  j_2 < \dotsc <j_i \leq
  n-1}e^{-t(\lambda_{j_1}+\dotsc+\lambda_{j_i})}\right) =
\frac1n\prod_{i=1}^{n-1}(1-e^{-\lambda_i t}).\qedhere
\]
\end{proof}


\begin{proof}[Proof of Theorem \ref{thm:chuk}]
The partitions that appear in lemma \ref{lemma:formula} are of the
form $[a,b,1^c]$, where $a+b+c=n$, $a\geq b> 0$, and $c \geq 0$. For
such a partition a simple calculation with the hook formula \cite[\S 4.12]{FH91} gives
$$ \dim U_{\lambda} = \frac {b(a-b+1)}{(b+c)(a+c+1)} \frac {n!}{a!b!c!} \leq \binom{n}{a,b,c} $$
Hence, the total number of summands in lemma \ref{lemma:formula} is bounded by $\sum_{a+b+c=n} \binom{n}{a,b,c}=3^n$.
Also, by the celebrated Caputo-Liggett-Richthammer theorem \cite{CLR09}, we have for all $j$,
$$\lambda_j(A, [a,b,1^c]) \geq \lambda_1(A)$$
The result now follows from lemma \ref{lemma:formula}.
\end{proof}

\begin{rem*}
For a non-hook-shaped partition $\rho$, the eigenvalues
$\lambda(A,\rho)$ are, in general, not a function of the eigenvalues
of the graph $A$. Such examples exist for $n$ as low as 4. In other
words, one can find two \emph{isospectral} (weighted) graphs $A_1$,
$A_2$ with 4 vertices for which $\lambda(A_i,[2,2])$ differ. By lemma
\ref{lemma:formula}, these two isospectral graphs also have different
values for $\PP(s_{3}(t)=1)$ for general $t$. Such examples can be
found by constructing $A_2$ as a conjugation of $A_1$ (for a generic $A_1$)
by an orthogonal perturbation of the identity which preserves the vector $(1,\dotsc,1)$.
\end{rem*}

\section{T\'oth's conjecture}\label{sec:Toth}
Let us start by describing T\'oth's work on the quantum Heisenberg ferromagnet \cite{T93}. Building on earlier work by Conlon and Solovej, he found what physicists term a \emph{graphical representation} of the model, i.e.\ a rigorous translation to an (interacting) random walk question. Most relevant for us is T\'oth's formula for the \emph{spontaneous magnetization} $m(\beta)$ of the quantum Heisenberg ferromagnet at inverse temperature $\beta$. Let $c_\beta(0)$ be the size of the cycle of $0$ at time $\beta$ for the interchange process on $[-r,r]^3$. Then \cite[(5.2)]{T93}
\[
m(\beta)=\frac 12 \lim_{n\to\infty}\lim_{r\to\infty}
\frac{\mathbb E\Big(\mathbf 1\{c_\beta(0)>n\}2^{\sum_{k\ge 1} s_k(\beta)}\Big)}
{\mathbb E\big(2^{\sum_{k\ge 1} s_K(\beta)}\big)}
\]
(recall that $s_k(\beta)$ is the number of cycles of length $k$ at time $\beta$, so their sum is just the total number of cycles, again for the interchange process on $[-r,r]^3$). Notice the somewhat counterintuitive fact that the inverse temperature becomes the time in this representation. With this formula (which some readers might feel more convenient to simply take as the definition of $m(\beta)$), T\'oth's conjecture is
\begin{conjecture}\label{conj:QHF}
$m(\beta)$ admits a phase transition, i.e.\ there exists some $\beta_c$ such that $m(\beta)=0$ for $\beta<\beta_c$ and $m(\beta)>0$ for $\beta>\beta_c$.
\end{conjecture}
It is natural to try first to remove the weights and investigate only $\mathbb P(c_\beta(0)>n)$ (T\'oth himself hints that this might be an interesting toy model). One then gets the following:
\begin{conjecture}\label{conj:interchange}
The function
\[
\lim_{n\to\infty}\lim_{r\to\infty}\mathbb P(c_\beta(0)>n)
\]
Undergoes a phase transition in $\beta$: it is zero for $\beta<\beta_c$ (not necessarily the same $\beta_c$ as in the previous conjecture) and positive for $\beta>\beta_c$.
\end{conjecture}

For both conjectures, it is not difficult to show that for $\beta$
sufficiently small the corresponding limits are zero. What is wide
open, for both conjectures, is that for $\beta$ sufficiently large,
the limits are non-zero. In other words, the big open problem at this
point is not sharpness or uniqueness of the phase transition, but the
actual existence of the high $\beta$ phase (the so-called ordered
phase).

Conjecture \ref{conj:interchange} was investigated when $[-r,r]^3$ is
replaced by the complete graph, the so-called \emph{mean-field}
case. The mean-field case was solved first by Berestycki \& Durrett
\cite{BD06} (who arrived at this problem from a different angle) and
then by Schramm \cite{S05}, who gave much more information on the
structure of the large cycles. In the mean-field case, $\beta_c$ is
explicitly known. An analog of conjecture \ref{conj:interchange} for
infinite graphs was investigated for trees
\cite{A03,H12a,H12b}. Notably, for trees of sufficiently high degree,
\cite{H12b} shows that there is a phase transition without calculating
the value of $\beta_c$.

We consider our Theorem \ref{thm:k} as a stepping stone for a
representation-theoretic attack on both conjectures. For conjecture
\ref{conj:interchange}, it reduces the problem to a calculation or
estimate of the eigenvalues of only some representations. In the
mean-field case, these eigenvalues are explicitly known \cite{DS81}
which leads to a simple analysis of this problem, see \cite{BK}. For conjecture \ref{conj:QHF}, this requires an extra ingredient even in the mean-field case: the interaction between the function $c_\beta(0)$ and the function $2^{\sum s_k(\beta)}$. We hope to tackle this problem in the future.

To gain some more insight on the non-mean-field case in conjecture \ref{conj:interchange}, let us examine the case $k=n$, i.e.\ apply Theorem \ref{thm:n} to the graph $[-r,r]^3$. For this graph the eigenfunctions and eigenvalues are explicitly known. Every vector $\xi\in\{0,\dotsc,2r-1\}^3$ the function $f(v)=\exp(2\pi i\langle\xi,v\rangle/(2r-1))$ is an eigenvector with the eigenvalue being $\sum(1-\cos(2\pi\xi_j/(2r-1)))$. Plugging these values into Theorem \ref{thm:n} with a little calculation shows, for example,
\[
\min\Big\{t:\PP(s_n(t)=1)\ge \frac 1{2n}\Big\}\approx r^2\approx n^{2/3}.
\]
In other words, the probability starts approaching the limit value $\frac 1n$ only when t is of the order of $n^{2/3}$ (for general dimensions, i.e.\ the graphs $[-r,r]^d$, the value would be $n^{2/d}$).

Thus we see that, unlike what one would expect from a naive extrapolation of conjecture \ref{conj:interchange}, the probability of a cycle of length $n$ does \emph{not} equilibrate at constant time but after much longer time.
The culprit
for this slow equilibration lies in the representation $[n-1,1]$
appearing in the sum when $k=n$. It is therefore reassuring to notice
that this representation appears \emph{only} when $k=n$. Again, at
this point our estimates for the eigenvalues
$\lambda_j([-r,r]^3,\rho)$ are too weak to give good information on
T\'oth's conjecture. See \cite{T10} for more information on these eigenvalues.


\subsection*{Acknowledgements}
We wish to thank Nati Linial for asking what happens when $t\to
0$; Richard Stanley for referring us to \cite{EH04}; and Yuval Roichman
and Ron Adin for interesting discussions. GK's research partially supported by the Israel Science
Foundation.

\appendix
\section*{Appendix. Kirchoff's matrix-tree theorem}

Here we give a new proof of the following old theorem, essentially due to
Kirchoff.
\begin{theorem*}Let $G$ be any weighted graph, and denote by $w_e$ the
  weight of the edge $e$. For a spanning tree $T$, denote
  $w(T)=\prod_{e\in T}w_e$ where the product is over all edges $e$ of
  $T$. Finally denote by $0=\lambda_0\le\lambda_1\le\dotsb\le\lambda_{n-1}$ the
  eigenvalues of the continuous time Laplacian $\Delta_G$. Then
\[
\sum_T w(T)=\frac{1}{n}\prod_{i=1}^{n-1}\lambda_i
\]
\end{theorem*}
(it is quite common to replace the product on the right-hand side by the absolute
value of a cofactor of the Laplacian $\Delta_G$, which gives an equivalent formulation, but
for the approach here this formulation is the more natural one).

\begin{proof}
Apply Theorem \ref{thm:n} with $t\to 0$. On the right-hand side one
gets
\begin{equation}\label{eq:rhs}
\frac{1}{n}\prod_{i=1}^{n-1}(1-e^{-\lambda_it})=\frac{1}{n}\prod_{i=1}^{n-1}(\lambda_it+O(t^2))
=
\frac{1}{n}t^{n-1}\prod_{i=1}^{n-1}\lambda_i + O(t^n).
\end{equation}
To estimate the left-hand side we use the coagulation-fragmentation
view of the interchange process, see e.g.\ \cite{S05}. By this we mean
the observation that when one applies the transposition $(i,j)$, if
$i$ and $j$ belong to different cycles in the permutations, then the
application of $(i,j)$ causes the cycles to merge; while if $i$ and
$j$ belong to the same cycle, then this causes the cycle to split. In
particular, if one draws an auxiliary graph $A_t$ with an edge between
every $i$ and $j$ for which the transposition $(i,j)$ was applied by
time $t$, then the cycles of the permutation at time $t$ are subsets
of the connected components of $A_t$.

Now, since we are interested in the case that $s_n(t)=1$, i.e.\ in the
case that the permutation is one big cycle, then this can happen only
when $A_t$ is connected. But a connected graph with $n$ vertices must
have at least $n-1$ edges, and if it has $n-1$ edges precisely then it
is a spanning tree. Further, if for some spanning tree $T$ the edges
of $T$ are exactly those that have rung by time $t$, and each one rang
exactly once, then the permutation is one big cycle, because a fragmentation
event never happened (these require closed paths in the graph $A_t$)
and $n-1$ coagulations lead to one big cycle. Hence we get for the
left-hand side,
\[
\PP(s_n(t)=1)=\sum_T \PP(\mbox{the edges of $T$ are exactly those that
  rang by time $t$})+O(t^n).
\]
Now, for an edge $e$ the probability that it rang exactly once by time
$t$ is $w_et+O(t^2)$. Further, all these events (for various $e$) are
independent. So we can continue to write
\[
\PP(s_n(t)=1)=\sum_T \prod_{e\in
  T}(w_et+O(t^2))+O(t^n)=t^{n-1}\sum_T\prod_{e\in T}w_e+O(t^n).
\]
Comparing to  (\ref{eq:rhs}) we get Kirchoff's theorem.
\end{proof}

\end{document}